\documentclass{article}
\usepackage[final]{graphicx}
\usepackage{psfrag}
\usepackage{amsmath,amsfonts,amssymb,amsxtra,subeqnarray}

\newlength{\extramargin}
\newcommand {\Real}{\ensuremath{{\mathbb{R}}}}

\newcommand {\Complex}{\ensuremath{{\mathbb{C}}}}

\newcommand{\D}{\ensuremath{\mathcal D}}

\newcommand{\I}{\ensuremath{\mathcal I}}
\newcommand{\J}{\ensuremath{\mathcal J}}
\newcommand{\V}{\ensuremath{\mathcal V}}

\newcommand{\setS}{\ensuremath{\mathcal S}}

\newcommand{\setE}{\ensuremath{\mathcal E}}

\newcommand{\N}{\ensuremath{\mathcal N}}

\newtheorem{theorem}{Theorem}

\newtheorem{corollary}{Corollary}
\newtheorem{assumeprime}{Assumption}

\newtheorem{lemma}{Lemma}

\newtheorem{remark}{Remark}
\newtheorem{assumption}{Assumption}
\newtheorem{example}{Example}

\newenvironment{proof}{\noindent {\bf Proof.}}{\hfill \hspace*{1pt}\hfill$\blacksquare$}

\begin{document}
\title{Synchronization of harmonic oscillators under restorative coupling with applications in electrical networks}
\author{S. Emre Tuna\footnote{The author is with Department of
Electrical and Electronics Engineering, Middle East Technical
University, 06800 Ankara, Turkey. Email: {\tt
tuna@eee.metu.edu.tr}}} \maketitle

\begin{abstract}
The role of restorative coupling on synchronization of coupled
identical harmonic oscillators is studied. Necessary and
sufficient conditions, under which the individual systems'
solutions converge to a common trajectory, are presented. Through
simple physical examples, the meaning and limitations of the
theorems are expounded. Also, to demonstrate their versatility,
the results are extended to cover LTI passive electrical networks.
One of the extensions generalizes the well-known link between the
asymptotic stability of the synchronization subspace and the
second smallest eigenvalue of the Laplacian matrix.
\end{abstract}

\section{Introduction}
Studying the collective behavior of coupled harmonic oscillators
has been a rewarding enterprize for researchers who try to enhance
their understanding on a much-encountered phenomenon in nature:
synchronization. For instance, it has been observed that two or
more identical pendulums\footnote{We restrict our attention to the
small oscillations, where the pendulum can be represented by a
linear model.} connected by means of dampers eventually swing in
unison even if initially they are not synchronized; see
Fig.~\ref{fig:pendula}. This outcome is not difficult to reach by
intuition. Since the energy of the system can only leak out
through the dampers, the pendulums should eventually settle to a
constant energy state where there is no leakage. No leakage
implies that the relative velocities are all zero. In other words,
all the pendulums are moving at equal velocities at all times.
This is only possible when they are synchronized.

\begin{figure}[h]
\begin{center}
\includegraphics[scale=0.55]{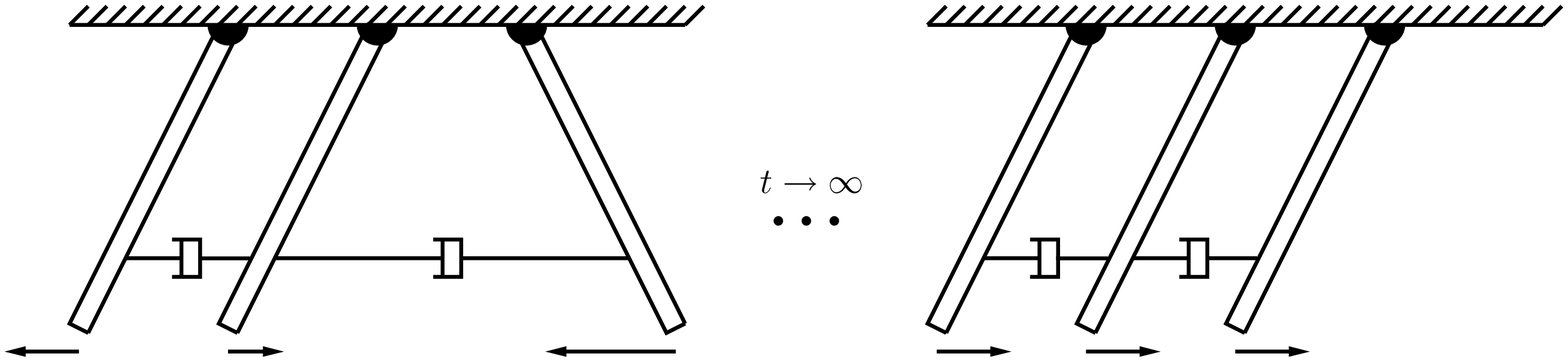}
\caption{Damper-coupled pendulums.} \label{fig:pendula}
\end{center}
\end{figure}

The simple example above has served as a starting point for many
significant generalizations. In \cite{ren08} Ren studies
synchronization of coupled harmonic oscillators allowing
time-varying oscillator dynamics as well as time-varying and
asymmetrical dampers. The case where the damping between a pair of
oscillators becomes effective only when the two are close enough
is investigated in \cite{su09}. The effect of nonlinear damping is
analyzed in \cite{cai10} and of impulsive damping in
\cite{zhou12}. A sampled-data approach is adopted in
\cite{zhang12,sun14}. Adaptive damping is covered in \cite{su13}
and synchronization in the presence of noisy damping is considered
in \cite{sun15}. Note that all these works consider only
dissipative coupling (e.g. dampers). From the engineering point of
view this choice is not surprising because introducing restorative
coupling (e.g. springs) will in general deteriorate performance by
causing longer and more oscillatory transient behavior; for
instance, the simulation results show that the three pendulums in
Fig.~\ref{fig:restorative} synchronize much less rapidly than
those in Fig.~\ref{fig:pendula}. Perhaps this may partly explain
why collective behavior of spring-coupled oscillators has
attracted more physicists than engineers. While for the engineer
spring is an option to couple two units, for the physicist it
represents an inherent characteristic of interaction. Relevant
investigations in the physics community goes as far back, if not
further, as the work of Fermi et al. \cite{fermi55} where chain of
nonlinearly coupled oscillator-like particles were studied. Due to
the richness of the subject and the increasing variety of
applications in both inanimate and biological systems, the area
has maintained its livelihood throughout many decades. See, for
instance, \cite{kapitaniak14}, \cite{marcheggiani14},
\cite{adato13}, \cite{kapitaniak14special} for recent progress.

\begin{figure}[h]
\begin{center}
\includegraphics[scale=0.55]{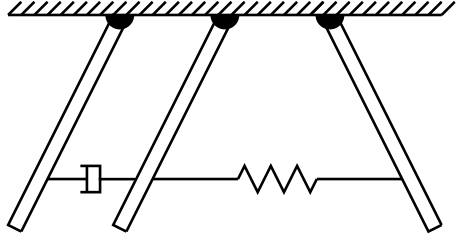}
\caption{Damper- and spring-coupled pendulums.}
\label{fig:restorative}
\end{center}
\end{figure}

Through this paper we aim to provide a comprehensive analysis of
the collective behavior of identical harmonic oscillators coupled
by both restorative and dissipative components. To the best of our
knowledge a detailed treatment of synchronization has not yet been
reported for this setting, where two different interconnection
graphs are simultaneously at work: the graph representing
restorative coupling and the graph representing dissipative
coupling. We present a necessary and sufficient condition on the
associated pair of Laplacian matrices, under which the individual
systems tend to oscillate in unison. We also point out a certain
sufficient-only, yet easier-to-check set of conditions
guaranteeing synchronization and exercise them on some simple
real-world examples for clarity. Later, we attempt to extend our
approach to the analysis of linear electrical networks of
identical oscillators (of arbitrary order) coupled through passive
impedances. For such networks we establish a link between
synchronization and the eigenvalues of the (complex) node
admittance matrix. This seems to be a natural extension of the
well-known connectivity condition in terms of the second smallest
eigenvalue of the (real-valued) Laplacian matrix.

\section{Coupled harmonic oscillators}\label{sec:mexicansky}

Consider the array of $q$ coupled harmonic
oscillators
\begin{eqnarray}\label{eqn:harmonic}
{\ddot z}_{i}+\omega_{0}^{2}z_{i}+\sum_{j=1}^{q}d_{ij}({\dot
z}_{i}-{\dot z}_{j})+\sum_{j=1}^{q}r_{ij}(z_{i}-z_{j})=0\,,\qquad
i=1,\,2,\,\ldots,\,q
\end{eqnarray}
where $z_{i}\in\Real$ and $\omega_{0}>0$ is the frequency of
uncoupled oscillations. The symmetric weights $d_{ij}=d_{ji}\geq
0$ and $r_{ij}=r_{ji}\geq 0$ respectively represent the
dissipative and restorative coupling between the $i$th and $j$th
oscillators. Note that without symmetry, i.e., either $d_{ij}\neq
d_{ji}$ or $r_{ij}\neq r_{ji}$, the solutions are not guaranteed
to be bounded unless some extra assumption is made. We take
$d_{ii}=0$ and $r_{ii}=0$. In this section and next we search for
conditions on the triple $(\omega_{0},\,\{d_{ij}\},\,\{r_{ij}\})$
under which the harmonic oscillators~\eqref{eqn:harmonic}
synchronize, i.e., $|z_{i}(t)-z_{j}(t)|\to 0$ as $t\to\infty$ for
all $i,\,j$ and all initial conditions.

Let $D,\,R\in\Real^{q\times q}$ denote the weighted Laplacian
matrices associated to the topologies described by the dissipative
coupling $\{d_{ij}\}$ and the restorative coupling $\{r_{ij}\}$,
respectively. That is,
\begin{eqnarray*}
D =
\left[\begin{array}{cccc}\sum_{j}d_{1j}&-d_{12}&\cdots&-d_{1q}\\
-d_{21}&\sum_{j}d_{2j}&\cdots&-d_{2q}\\
\vdots&\vdots&\ddots&\vdots\\
-d_{q1}&-d_{q2}&\cdots&\sum_{j}d_{qj}
\end{array}\right]\,,\qquad
R =
\left[\begin{array}{cccc}\sum_{j}r_{1j}&-r_{12}&\cdots&-r_{1q}\\
-r_{21}&\sum_{j}r_{2j}&\cdots&-r_{2q}\\
\vdots&\vdots&\ddots&\vdots\\
-r_{q1}&-r_{q2}&\cdots&\sum_{j}r_{qj}
\end{array}\right]\,.
\end{eqnarray*}
Note that these matrices are symmetric positive semidefinite since
$d_{ij}=d_{ji}\geq 0$ and $r_{ij}=r_{ji}\geq 0$. In particular, we
can write $z^{T}Dz=\sum_{j>i}d_{ij}(z_{i}-z_{j})^2$ and
$z^{T}Rz=\sum_{j>i}r_{ij}(z_{i}-z_{j})^2$, where $z=[z_{1}\ z_{2}\
\cdots\ z_{q}]^{T}\in\Real^{q}$. Let us now rewrite
\eqref{eqn:harmonic} as
\begin{eqnarray*}
{\ddot z}+\omega_{0}^{2}z+D{\dot z}+Rz=0\,.
\end{eqnarray*}
This, using $x=[z^{T}\ {\dot z}^{T}]^{T}\in\Real^{2q}$, allows us
to obtain
\begin{eqnarray}\label{eqn:harmonicarray}
{\dot
x}=\left[\begin{array}{cc}0&I_{q}\\-(\omega_{0}^{2}I_{q}+R)&-D\end{array}\right]x=:\Phi
x
\end{eqnarray}
where $I_{q}\in\Real^{q\times q}$ is the identity matrix.
Employing the symmetric positive definite matrix
\begin{eqnarray*}
P=\frac{1}{2}
\left[\begin{array}{cc}\omega_{0}^{2}I_{q}+R&0\\0&I_{q}\end{array}\right]
\end{eqnarray*}
we can establish the following Lyapunov equality
\begin{eqnarray*}
\Phi^{T}P+P\Phi=-\left[\begin{array}{cc}0&0\\0&D\end{array}\right]\,.
\end{eqnarray*}
Since the righthand side is negative semidefinite, each solution
$x(t)$ of the system~\eqref{eqn:harmonicarray} is bounded.
Moreover, by Krasovskii-LaSalle principle, $x(t)$ should converge
to the largest invariant region contained in the intersection
$\D\cap\{x:x^{T}Px\leq x(0)^{T}Px(0)\}$ where
\begin{eqnarray*}
\D:=\left\{x:\left[\begin{array}{cr}0&0\\0&D\end{array}\right]x=0\right\}\,.
\end{eqnarray*}
It turns out that the condition
\begin{eqnarray}\label{eqn:PBH}
{\rm null}\left[\begin{array}{c}R-\lambda I_{q}\\
D\end{array}\right]\subset {\rm range}\,{\bf 1}_{q}\ \mbox{for
all}\ \lambda\in\Complex
\end{eqnarray}
(where ${\bf 1}_{q}\in\Real^{q}$ is the vector of all ones)
guarantees that this largest invariant region is contained in the
synchronization subspace
\begin{eqnarray*}
\setS:={\rm range}\,\left[\begin{array}{cc}{\bf 1}_{q}&0\\0&{\bf
1}_{q}\end{array}\right]\,.
\end{eqnarray*}
In other words:

\begin{lemma}\label{lem:one}
Let \eqref{eqn:PBH} hold. Then and only then
\begin{eqnarray}\label{eqn:mexicansky}
x(t)\in\D\ \mbox{for all}\ t\implies x(t)\in\setS\ \mbox{for all}\
t
\end{eqnarray}
where $x(t)$ is the solution of the
system~\eqref{eqn:harmonicarray}.
\end{lemma}

\begin{proof}
We first establish
\eqref{eqn:PBH}$\implies$\eqref{eqn:mexicansky}. Let
$x(t)=[z(t)^{T}\ {\dot z}(t)^{T}]^{T}$ be a solution of the
system~\eqref{eqn:harmonicarray} that identically belongs to $\D$.
This means $D\dot{z}(t)\equiv 0$. Also,
\begin{eqnarray*}
{\dot x}
=\left[\begin{array}{cc}0&I_{q}\\-(\omega_{0}^{2}I_{q}+R)&0\end{array}\right]x-\left[\begin{array}{cc}0&0\\
0&D\end{array}\right]x=\left[\begin{array}{cc}0&I_{q}\\-(\omega_{0}^{2}I_{q}+R)&0\end{array}\right]x
\end{eqnarray*}
which implies
\begin{eqnarray}\label{eqn:nominal}
{\ddot z}+(\omega_{0}^{2}I_{q}+R)z=0\,.
\end{eqnarray}
Let $\lambda_{1},\,\lambda_{2},\,\ldots,\,\lambda_{p}$ be the
distinct ($p\leq q$) eigenvalues of $R$. Since $R$ is symmetric
positive semidefinite, these eigenvalues are real and nonnegative.
Consequently, the matrix $[\omega_{0}^{2}I_{q}+R]$ is symmetric
positive definite with eigenvalues
$\omega_{0}^{2}+\lambda_{1},\,\omega_{0}^{2}+\lambda_{2},\,\ldots,\,\omega_{0}^{2}+\lambda_{p}$.
Therefore \eqref{eqn:nominal} implies that the solution has the
form \cite[\S 23]{arnold89}
\begin{eqnarray}\label{eqn:arnold}
z(t)={\rm Re}\sum_{k=1}^{p}e^{j\omega_{k}t}\xi_{k}
\end{eqnarray}
where $\omega_{k}=\sqrt{\omega_{0}^{2}+\lambda_{k}}$ are distinct
and positive, and each $\xi_{k}\in\Complex^{q}$ (some of which may
be zero) satisfies
\begin{eqnarray}\label{eqn:arnold2}
0&=&([\omega_{0}^{2}I_{q}+R]-\omega_{k}^{2}I_{q})\xi_{k}\nonumber\\
&=&(R-\lambda_{k}I_{q})\xi_{k}\,.
\end{eqnarray}
Note that the condition $D\dot{z}(t)\equiv 0$ and
\eqref{eqn:arnold} imply
\begin{eqnarray}\label{eqn:arnold3}
D\xi_{k}=0
\end{eqnarray}
since $\omega_{k}$ are distinct and nonzero. Combining
\eqref{eqn:arnold2} and \eqref{eqn:arnold3} we can write
\begin{eqnarray}\label{eqn:PBH2}
\xi_{k}\in{\rm null}\left[\begin{array}{c}R-\lambda_{k} I_{q}\\
D\end{array}\right]\,.
\end{eqnarray}
Suppose now \eqref{eqn:PBH} holds. Then \eqref{eqn:PBH2} implies
$\xi_{k}\in{\rm range}\,{\bf 1}_{q}$ for all $k$. By
\eqref{eqn:arnold} this readily yields $z(t)\in{\rm range}\,{\bf
1}_{q}$ for all $t$. Consequently, ${\dot z}(t)\in{\rm
range}\,{\bf 1}_{q}$ and $x(t)\in\setS$ for all $t$.

Now we show the other direction
\eqref{eqn:mexicansky}$\implies$\eqref{eqn:PBH}. Suppose that
condition~\eqref{eqn:PBH} is not true. Then we can find an
eigenvalue $\lambda_{*}\geq 0$ of $R$ and a nonzero vector
$\xi_{*}\in\Real^{q}$ satisfying $\xi_{*}\notin{\rm range}\, {\bf
1}_{q}$ such that $D\xi_{*}=0$ and $(R-\lambda_{*}
I_{q})\xi_{*}=0$. Let
$\omega_{*}=\sqrt{\omega_{0}^{2}+\lambda_{*}}$. Using the pair
$(\omega_{*},\,\xi_{*})$ let us construct the function $
z_{*}:\Real\to\Real^{q}$ as $z_{*}(t)={\rm Re}(e^{j\omega_{*}
t}\xi_{*})$. This function satisfies the following properties.
First, since $\xi_{*}\notin{\rm range}\,{\bf 1}_{q}$, we have
\begin{eqnarray}\label{eqn:vball1}
z_{*}(0)=\xi_{*}\notin{\rm range}\,{\bf 1}_{q}\,.
\end{eqnarray}
Second, since $D\xi_{*}=0$, we have at all times
\begin{eqnarray}\label{eqn:vball2}
D\dot{z}_{*}(t)={\rm Re}(j\omega_{*}e^{j\omega_{*}
t}D\xi_{*})=0\,.
\end{eqnarray}
Third, since $[R+(\omega_{0}^2-\omega_{*}^{2})I_{q}]\xi_{*}=0$, we
can write at all times
\begin{eqnarray*}
{\ddot
z}_{*}(t)+(\omega_{0}^{2}I_{q}+R)z_{*}(t)=-\omega_{*}^{2}z_{*}(t)+(\omega_{0}^{2}I_{q}+R)z_{*}(t)={\rm
Re}(e^{j\omega_{*}
t}[R+(\omega_{0}^2-\omega_{*}^{2})I_{q}]\xi_{*})=0
\end{eqnarray*}
which together with \eqref{eqn:vball2} leads to
\begin{eqnarray}\label{eqn:vball3}
{\ddot
z}_{*}(t)+D\dot{z}_{*}(t)+(\omega_{0}^{2}I_{q}+R)z_{*}(t)\equiv0\,.
\end{eqnarray}
Let $x_{*}(t)=[z_{*}(t)^{T}\ {\dot z}_{*}(t)^{T}]^{T}$. It follows
from \eqref{eqn:vball3} that $x_{*}(t)$ satisfies
\eqref{eqn:harmonicarray} and hence is a solution of the system.
By \eqref{eqn:vball2} we can assert that the solution $x_{*}(t)$
belongs identically to $\D$, but \eqref{eqn:vball1} tells us that
$x_{*}(t)$ does not identically belong to $\setS$. That is, the
condition~\eqref{eqn:mexicansky} fails.
\end{proof}

\vspace{0.12in}

To the question asked at the beginning we can now give the answer:

\begin{theorem}\label{thm:one}
The harmonic oscillators~\eqref{eqn:harmonic} synchronize if and
only if \eqref{eqn:PBH} holds.
\end{theorem}

\begin{remark}
Note that the condition~\eqref{eqn:PBH} does not depend on the
natural frequency $\omega_{0}$.
\end{remark}

Before we end this section we attempt to interpret
condition~\eqref{eqn:PBH}. Recall that, given matrices
$C\in\Real^{m\times n}$ and $A\in\Real^{n\times n}$, the
unobservable subspace of the pair $(C,\,A)$ is
\begin{eqnarray*}
{\rm unobs}\,(C,\,A)={\rm null}\left[\begin{array}{c}C\\
CA\\
\vdots\\
CA^{n-1}\end{array}\right]\,.
\end{eqnarray*}
The below result reveals the meaning of condition~\eqref{eqn:PBH}
from the observability point of view.
\begin{theorem}\label{thm:observe}
Condition~\eqref{eqn:PBH} holds if and only if ${\rm
unobs}\,(D,\,R)={\rm range}\,{\bf 1}_{q}$.
\end{theorem}

\begin{proof}
Suppose \eqref{eqn:PBH} fails. Then we can find an eigenvector
$\xi$ of $R$ satisfying $\xi\notin{\rm range}\,{\bf 1}_{q}$ and
$D\xi=0$. Let $\lambda$ be the corresponding eigenvalue, i.e.,
$R\xi=\lambda\xi$. We can write
\begin{eqnarray*}
\left[\begin{array}{c}D\\
DR\\
\vdots\\
DR^{q-1}\end{array}\right]\xi =
\left[\begin{array}{c}D\xi\\
DR\xi\\
\vdots\\
DR^{q-1}\xi\end{array}\right] =
\left[\begin{array}{c}D\xi\\
\lambda D\xi\\
\vdots\\
\lambda^{q-1}D\xi\end{array}\right]=0\,.
\end{eqnarray*}
Therefore $\xi\in{\rm unobs}\,(D,\,R)$. Since $\xi\notin{\rm
range}\,{\bf 1}_{q}$, we must have ${\rm unobs}\,(D,\,R)\neq{\rm
range}\,{\bf 1}_{q}$.

Now we show the other direction. Suppose ${\rm
unobs}\,(D,\,R)\neq{\rm range}\,{\bf 1}_{q}$. By definition we
have $D{\bf 1}_{q}=0$ and $R{\bf 1}_{q}=0$, meaning ${\rm
unobs}\,(D,\,R)\supset{\rm range}\,{\bf 1}_{q}$. Consequently,
$2\leq{\rm dim}\,{\rm unobs}\,(D,\,R)=:\ell$. Let
$\{\eta_{1},\,\eta_{2},\,\ldots,\,\eta_{\ell}\}$ with
$\eta_{1}={\bf 1}_{q}$ be an orthogonal basis for ${\rm
unobs}\,(D,\,R)$. Since $R$ is symmetric its eigenvectors form an
orthogonal basis for $\Complex^{q}$. Let this basis be
$\{\xi_{1},\,\xi_{2},\,\ldots,\,\xi_{q}\}$ with $\xi_{1}={\bf
1}_{q}$. Note that we have ${\rm
span}\,\{\eta_{2},\,\ldots,\,\eta_{\ell}\}\subset{\rm
span}\,\{\xi_{2},\,\ldots,\,\xi_{q}\}$. Now, let us choose an
arbitrary nonzero vector $w\in{\rm
span}\,\{\eta_{2},\,\ldots,\,\eta_{\ell}\}$. Since ${\rm
unobs}\,(D,\,R)$ is $R$-invariant we have $Rw\in{\rm
span}\,\{\eta_{1},\,\ldots,\,\eta_{\ell}\}$. Moreover, $w\in{\rm
span}\,\{\xi_{2},\,\ldots,\,\xi_{q}\}$ implies $Rw\in{\rm
span}\,\{\xi_{2},\,\ldots,\,\xi_{q}\}$ because $\xi_{i}$ are
eigenvectors. Hence we can write
\begin{eqnarray*}
Rw&\in&{\rm span}\,\{\eta_{1},\,\ldots,\,\eta_{\ell}\}\cap{\rm
span}\,\{\xi_{2},\,\ldots,\,\xi_{q}\}\\
&=&{\rm span}\,\{\eta_{2},\,\ldots,\,\eta_{\ell}\}\,.
\end{eqnarray*}
This implies (since $w$ was arbitrary) that ${\rm
span}\,\{\eta_{2},\,\ldots,\,\eta_{\ell}\}$ is $R$-invariant.
Consequently, ${\rm span}\,\{\eta_{2},\,\ldots,\,\eta_{\ell}\}$
contains at least one eigenvector $\xi$ of $R$. It must be that
$\xi\notin{\rm range}\,{\bf 1}_{q}$ because ${\bf 1}_{q}\notin{\rm
span}\,\{\eta_{2},\,\ldots,\,\eta_{\ell}\}$. Let $\lambda$ be the
corresponding eigenvalue, i.e., $R\xi=\lambda\xi$. Since
$\xi\in{\rm unobs}\,(D,\,R)$ we have $\xi\in{\rm null}\,D$
yielding ${\rm null}\,D\cap{\rm null}\,(R-\lambda
I_{q})\supset{\rm span}\,\{\xi\}$ which implies that
$\eqref{eqn:PBH}$ fails to hold.
\end{proof}

\section{Sufficient conditions for synchronization}

Although condition~\eqref{eqn:PBH} tells us definitely whether a
given array of harmonic oscillators will synchronize or not, it
may nevertheless be expensive or simply impossible to employ when,
for instance, the number of oscillators is large or certain
parameter values are unknown. Therefore it is worthwhile to look
for sufficient-only, yet simpler-to-check conditions to determine
synchronization. This is what we intend to do in this section.

Recall that an undirected graph is a pair of sets $(\V,\,\setE)$
where $\V=\{v_{1},\,v_{2},\,\ldots,\,v_{q}\}$ is the set of
vertices and the elements of the (possibly empty) set $\setE$ are
some (unordered) pairs of vertices $(v_{i},\,v_{j})$. Let us now
introduce two graphs associated to the array of harmonic
oscillators~\eqref{eqn:harmonic} as follows. The graph
$\Gamma_{\rm d}=(\V,\,\setE_{\rm d})$ describes the
interconnection associated to dissipative coupling and is such
that $(v_{i},\,v_{j})\in\setE_{\rm d}$ when $d_{ij}\neq 0$.
Similarly, $\Gamma_{\rm r}=(\V,\,\setE_{\rm r})$ denotes the
restorative coupling topology and $(v_{i},\,v_{j})\in\setE_{\rm
r}$ when $r_{ij}\neq 0$.

By construction $D{\bf 1}_{q}=0$, which yields ${\rm
null}\,D\supset{\rm range}\,{\bf 1}_{q}$. Since a graph is
connected when the eigenvalue of the associated Laplacian matrix
at the origin is simple, we have ${\rm null}\,D={\rm range}\,{\bf
1}_{q}$ when $\Gamma_{\rm d}$ is connected. Note that
\eqref{eqn:PBH} is trivially satisfied if ${\rm null}\,D={\rm
range}\,{\bf 1}_{q}$. Therefore we can assert:

\begin{corollary}\label{cor:simple}
The harmonic oscillators~\eqref{eqn:harmonic} synchronize if the
dissipative coupling graph $\Gamma_{\rm d}$ is connected.
\end{corollary}

As mentioned earlier, a collection of identical pendulums
connected (only) by dampers eventually synchronize. What the above
result adds to this statement is that even if we supplement the
collection by springs connecting some pairs of pendulums, the
tendency for synchronization cannot be destroyed.
Corollary~\ref{cor:simple} is hardly surprising. Now we move on to
establishing a less evident result. We begin by defining the
matrix $R_{\Delta}\in\Real^{q\times q}$ as
\begin{eqnarray*}
R_{\Delta} =
\left[\begin{array}{cccc}\sum_{j}{\hat r}_{1j}&-{\hat r}_{12}&\cdots&-{\hat r}_{1q}\\
-{\hat r}_{21}&\sum_{j}{\hat r}_{2j}&\cdots&-{\hat r}_{2q}\\
\vdots&\vdots&\ddots&\vdots\\
-{\hat r}_{q1}&-{\hat r}_{q2}&\cdots&\sum_{j}{\hat r}_{qj}
\end{array}\right]\qquad\mbox{where}\qquad {\hat r}_{ij}=\left\{\begin{array}{ccl}
r_{ij}&\mbox{for}&d_{ij}=0\,,\\
0&\mbox{for}&d_{ij}\neq 0\,.
\end{array}\right.
\end{eqnarray*}
We let $\Gamma_{\Sigma}=(\V,\,\setE_{\Sigma})$ where
$\setE_{\Sigma}=\setE_{\rm r}\cup\setE_{\rm d}$. Likewise,
employing the set difference $\setE_{\Delta}=\setE_{\rm
r}-\setE_{\rm d}$ we define the graph
$\Gamma_{\Delta}=(\V,\,\setE_{\Delta})$. Let $\Gamma_{\Delta}$
have $c$ connected components, which we denote by
$\Gamma_{\ell}=(\V_{\ell},\,\setE_{\ell})$ for
$\ell=1,\,2,\,\ldots,\,c$. By definition all the pairs
$(\V_{\ell},\,\V_{k})$ and $(\setE_{\ell},\,\setE_{k})$ are
disjoint for $\ell\neq k$. Moreover, $\bigcup_{\ell}\V_{\ell}=\V$
and $\bigcup_{\ell}\setE_{\ell}=\setE_{\Delta}$. Note that if
$\Gamma_{\Delta}$ itself is connected then $c=1$ and
$\Gamma_{1}=\Gamma_{\Delta}$. Let $n_{\ell}=|\V_{\ell}|$ be the
number of vertices that belong to $\Gamma_{\ell}$. Without loss of
generality let the vertices $v_{i}$ be such labelled that
$\V_{1}=\{v_{1},\,v_{2},\,\ldots,\,v_{n_{1}}\}$,
$\V_{2}=\{v_{n_{1}+1},\,v_{n_{1}+2},\,\ldots,\,v_{n_{1}+n_{2}}\}$,
and so on. Then $R_{\Delta}$ has the block diagonal form
\begin{eqnarray*}
R_{\Delta}=\left[\begin{array}{cccc}R_{1}&0&\cdots&0\\
0&R_{2}&\cdots&0\\
\vdots&\vdots&\ddots&\vdots\\
0&0&\cdots&R_{c}
\end{array}\right]
\end{eqnarray*}
with $R_{\ell}\in\Real^{n_{\ell}\times n_{\ell}}$. Note that each
$R_{\ell}$ is symmetric positive semidefinite and satisfies
$R_{\ell}{\bf 1}_{n_{\ell}}=0$ meaning there is an eigenvalue at
the origin. Since $\Gamma_{\ell}$ is connected this eigenvalue at
the origin is simple. Hence the eigenvalues of $R_{\ell}$ can be
ordered as
$0=\lambda_{1,\,\ell}<\lambda_{2,\,\ell}\leq\cdots\leq\lambda_{n_{\ell},\,\ell}$.
Now, for each $\ell=1,\,2,\,\ldots,\,c$, define the following
system
\begin{eqnarray}\label{eqn:omega}
{\ddot
\eta}_{\ell}+(\omega_{0}^{2}I_{n_{\ell}}+R_{\ell})\eta_{\ell}=0
\end{eqnarray}
with $\eta_{\ell}=[z_{\sigma_{\ell}+1}\ z_{\sigma_{\ell}+2}\
\cdots\ z_{\sigma_{\ell}+n_{\ell}}]^{T}\in\Real^{n_{\ell}}$ where
$\sigma_{1}=0$ and $\sigma_{\ell}=n_{1}+n_{2}+\cdots+n_{\ell-1}$
for $\ell\geq 2$. Note that we can write $z=[z_{1}\ z_{2}\ \cdots\
z_{q}]^{T}=[\eta_{1}^{T}\ \eta_{2}^{T}\ \cdots\
\eta_{c}^{T}]^{T}$. Let $\Omega_{\ell}=\{\omega>0:
\omega^{2}=\omega_{0}^{2}+\lambda_{k,\,\ell}\,,\
k=1,\,2,\,\ldots,\,n_{\ell}\}$ denote the set of characteristic
frequencies of the system~\eqref{eqn:omega}. Note that the
frequency $\omega_{0}$ belongs to every $\Omega_{\ell}$ because
$\lambda_{1,\,\ell}=0$. Now we list a sufficient set of conditions
guaranteeing synchronization.

\begin{assumption}[P]\label{assume:one}
The harmonic oscillators~\eqref{eqn:harmonic} satisfy the
following conditions.
\begin{enumerate}
\item For all $\ell=1,\,2,\,\ldots,\,c$ the
system~\eqref{eqn:omega} is observable from each $z_{k}$ for all
$k\in\{i:v_{i}\in\V_{\ell}\}$. \item
$\Omega_{1}\cap\Omega_{2}\cap\cdots\cap\Omega_{c}=\{\omega_{0}\}$\,.
\item $\Gamma_{\Sigma}$ is connected.
\end{enumerate}
\end{assumption}

A more mathematical (less physical) version of
Assumption~\ref{assume:one} reads:

\begin{assumeprime}[M]\label{assume:oneprime}
The harmonic oscillators~\eqref{eqn:harmonic} satisfy the
following conditions.
\begin{enumerate}
\item None of the matrices $R_{1},\,R_{2},\,\ldots,\,R_{c}$ has an
eigenvector with a zero entry. \item $\lambda=0$ is the only
common eigenvalue of the matrices
$R_{1},\,R_{2},\,\ldots,\,R_{c}$. \item ${\rm null}\,R\cap{\rm
null}\,D={\rm range}\,{\bf 1}_{q}$.
\end{enumerate}
\end{assumeprime}

\begin{theorem}
The harmonic oscillators~\eqref{eqn:harmonic} synchronize if
Assumption~\ref{assume:one} holds.
\end{theorem}

\begin{proof}
If we can show that Assumption~\ref{assume:one} implies
condition~\eqref{eqn:PBH} then by Theorem~\ref{thm:one} the
oscillators must synchronize. Let us establish the implication by
contradiction. Suppose that \eqref{eqn:PBH} is not true but
Assumption~\ref{assume:one} holds. Then we can find an eigenvector
$\xi\notin{\rm range}\,{\bf 1}_{q}$ satisfying $D\xi=0$ and
$(R-\lambda I_{q})\xi=0$ for some $\lambda$. This eigenvalue
$\lambda$ cannot be zero for then we have $\xi\in{\rm
null}\,R\cap{\rm null}\,D$ which contradicts the third condition
of Assumption~\ref{assume:one}. Let us therefore study the case
$\lambda\neq 0$ in the sequel.

Let us employ the partitions $\xi=[z_{1}\ z_{2}\ \cdots\
z_{q}]^{T}=[\eta_{1}^{T}\ \eta_{2}^{T}\ \cdots\
\eta_{c}^{T}]^{T}$. Since $D\xi=0$ we can write
$0=\xi^{T}D\xi=\sum_{j>i}d_{ij}(z_{i}-z_{j})^{2}$ which implies
that for a given pair $(i,\,j)$ of indices either $d_{ij}=0$ or
$z_{i}=z_{j}$. Since $d_{ij}=0$ means ${\hat r}_{ij}=r_{ij}$ we
have ${\hat r}_{ij}(z_{i}-z_{j})=r_{ij}(z_{i}-z_{j})$ for all
$i,\,j$. Now we can proceed to claim $R\xi=R_{\Delta}\xi$ because
\begin{eqnarray*}
R\xi=\left[\begin{array}{c}\sum_{j}r_{1j}(z_{1}-z_{j})\\
\vdots\\\sum_{j}r_{qj}(z_{q}-z_{j})\end{array}\right]=\left[\begin{array}{c}\sum_{j}{\hat r}_{1j}(z_{1}-z_{j})\\
\vdots\\\sum_{j}{\hat
r}_{qj}(z_{q}-z_{j})\end{array}\right]=R_{\Delta}\xi\,.
\end{eqnarray*}
Therefore $R_{\Delta}\xi=\lambda\xi$. Then the block diagonal
structure of $R_{\Delta}$ allows us to write
$R_{\ell}\eta_{\ell}=\lambda\eta_{\ell}$ for all
$\ell=1,\,2,\,\ldots,\,c$. Since by assumption nonzero eigenvalue
$\lambda$ is not common to all $R_{\ell}$ we should have
$\eta_{\ell}=0$ for at least one index $\ell$. Also, again by
assumption, no entry of $\eta_{\ell}$ can be zero whenever
$\eta_{\ell}\neq 0$. That is, if $\eta_{\ell}\neq 0$ then
$z_{k}\neq 0$ for all $k\in\{i:v_{i}\in\V_{\ell}\}$. Let us now
define two (nonempty) sets of indices
$\I:=\{i:v_{i}\in\V_{\ell},\,\eta_{\ell}\neq
0,\,\ell=1,\,2,\,\ldots,\,c\}$ and its complement
$\J:=\{1,\,2,\,\ldots,\,q\}-\I$. Note that $z_{i}\neq 0$ for
$i\in\I$ and $z_{i}=0$ for $i\in\J$. For any pair of indices
$(i,\,j)$ with $i\in\I$ and $j\in\J$ we can assert the following.
(i) $d_{ij}=0$. Because $z_{i}-z_{j}\neq 0$ and $0=\xi^{T}D\xi\geq
d_{ij}(z_{i}-z_{j})^{2}$. (ii) ${\hat r}_{ij}=0$. Because, by how
we constructed the sets $\I$ and $\J$, the vertices $v_{i}$ and
$v_{j}$ cannot belong to the same vertex set $\V_{\ell}$. Then by
the block diagonal form of $R_{\Delta}$ the entry ${\hat r}_{ij}$
must be zero. (iii) $r_{ij}=0$. Because $d_{ij}=0$ means
$r_{ij}={\hat r}_{ij}$.

Construct the vector ${\hat\xi}=[{\hat z}_{1}\ {\hat z}_{2}\
\cdots\ {\hat z}_{q}]^{T}$ with entries ${\hat z}_{i}=0$ for
$i\in\I$ and ${\hat z}_{i}=1$ for $i\in\J$. Clearly,
$\hat\xi\notin{\rm range}\,{\bf 1}_{q}$. We can write
\begin{eqnarray*}
{\hat\xi}^{T}D{\hat\xi}&=&\sum_{j>i}d_{ij}({\hat z}_{i}-{\hat
z}_{j})^{2}\\ &=&\frac{1}{2}\sum_{i,j\in\I}d_{ij}({\hat
z}_{i}-{\hat z}_{j})^{2}+\sum_{i\in\I,\,j\in\J}d_{ij}({\hat
z}_{i}-{\hat z}_{j})^{2}+\frac{1}{2}\sum_{i,j\in\J}d_{ij}({\hat
z}_{i}-{\hat
z}_{j})^{2}\\
&=&\frac{1}{2}\sum_{i,j\in\I}d_{ij}(0-0)^{2}+\frac{1}{2}\sum_{i,j\in\J}d_{ij}(1-1)^{2}\\
&=&0
\end{eqnarray*}
where we used the fact that $d_{ij}=0$ when $i\in\I$ and $j\in\J$.
Then ${\hat\xi}^{T}D{\hat\xi}=0$ implies $D{\hat\xi}=0$ because
$D$ is symmetric positive semidefinite. Since we also have that
$r_{ij}=0$ when $i\in\I$ and $j\in\J$, we can similarly establish
$R{\hat\xi}=0$. Hence $\hat\xi\in{\rm null}\,R\cap{\rm null}\,D$.
But $\hat\xi\notin{\rm range}\,{\bf 1}_{q}$. This contradicts the
third condition of Assumption~\ref{assume:one}.
\end{proof}

\vspace{0.12in}

As stated earlier, Assumption~\ref{assume:one} is only sufficient
for synchronization. Hence if an assembly of harmonic oscillators
fail to synchronize, at least one of the three conditions listed
therein must not hold. Of those three conditions, the necessity of
the third one (that $\Gamma_{\Sigma}$ is connected) is evident.
However, the relation of the remaining two conditions to
synchronization is subtle and requires due attention. To better
understand the meanings of those conditions we now provide two
examples, where harmonic oscillators do not synchronize. Each
example violates one of the first two conditions of
Assumption~\ref{assume:one}.

\begin{example}\label{ex:one}
Consider the following four coupled harmonic oscillators
\begin{eqnarray*}
&&{\ddot z}_{1}+\omega_{0}^{2}z_{1}+r(z_{1}-z_{2})=0\\
&&{\ddot z}_{2}+\omega_{0}^{2}z_{2}+d({\dot z}_{2}-{\dot z}_{4})+r(z_{2}-z_{1})+r(z_{2}-z_{3})=0\\
&&{\ddot z}_{3}+\omega_{0}^{2}z_{3}+r(z_{3}-z_{2})=0\\
&&{\ddot z}_{4}+\omega_{0}^{2}z_{4}+d({\dot z}_{4}-{\dot z}_{2})=0
\end{eqnarray*}
where $\omega_{0},\,d,\,r>0$. The associated $D$ and $R$ matrices
are provided below.
\begin{eqnarray*}
D=\left[\begin{array}{rrrr}0&0&0&0\\0&d&0&-d\\0&0&0&0\\0&-d&0&d\end{array}\right]
\,,\qquad
R=\left[\begin{array}{rrrr}r&-r&0&0\\-r&2r&-r&0\\0&-r&r&0\\0&0&0&0\end{array}\right]\,.
\end{eqnarray*}
It turns out that for the pair $(D,\,R)$ condition~\eqref{eqn:PBH}
fails to hold and hence by Theorem~\ref{thm:one} the oscillators
do not synchronize. In particular,
\begin{eqnarray*}
{\rm null}\left[\begin{array}{c}R-rI_{q}\\
D\end{array}\right] = {\rm range}
\left[\begin{array}{r}1\\0\\-1\\0\end{array}\right]\not\subset
{\rm range} \left[\begin{array}{r}1\\1\\1\\1\end{array}\right]\,.
\end{eqnarray*}
The solution corresponding to the eigenvector $[1\ 0\ -1\ 0]^{T}$
is shown in Fig.~\ref{fig:unobs} where the first and third
pendulums oscillate (with $\pi$ radians of phase difference) at
frequency $\omega=\sqrt{\omega_{0}^{2}+r}$ while the second and
fourth oscillators sit still. (We note that the two springs are
identical.)
\begin{figure}[h]
\begin{center}
\includegraphics[scale=0.55]{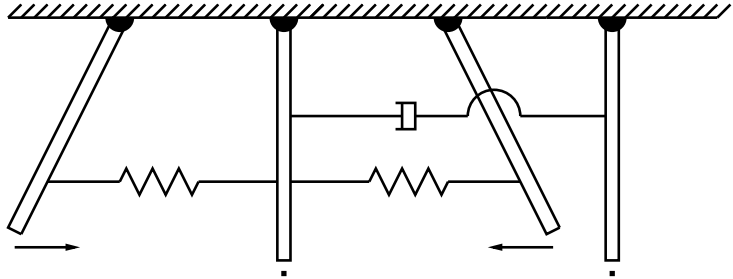}
\caption{Pendulums out of synchrony due to lack of observability.}
\label{fig:unobs}
\end{center}
\end{figure}
Let us now figure out which condition(s) of
Assumption~\ref{assume:one} is violated for our example. First we
consider the interconnection. The graphs $\Gamma_{\Sigma}$ and
$\Gamma_{\Delta}$ are given in Fig.~\ref{fig:graphs1}. Since the
graph $\Gamma_{\Sigma}$ is connected, the third condition of
Assumption~\ref{assume:one} is satisfied.
\begin{figure}[h]
\begin{center}
\includegraphics[scale=0.55]{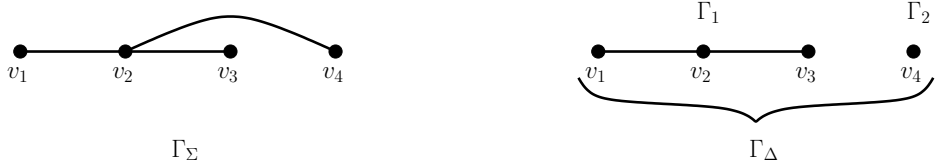}
\caption{The graphs associated to the coupled harmonic oscillators
in Example~\ref{ex:one}.}\label{fig:graphs1}
\end{center}
\end{figure}
Note that $\Gamma_{\Delta}$ has two components: $\Gamma_{1}$ and
$\Gamma_{2}$, the latter being a single vertex. Related to these
graphs are the matrices $R_{1}\in\Real^{3\times 3}$ and
$R_{2}\in\Real^{1\times 1}$. We have $R_{2}=0$ since $\Gamma_{2}$
has no edges. The matrix $R_{1}$ on the other hand has the
following form
\begin{eqnarray*}
R_{1}=\left[\begin{array}{rrr}r&-r&0\\-r&2r&-r\\0&-r&r\end{array}\right]\,.
\end{eqnarray*}
For the graph $\Gamma_{2}$ the system~\eqref{eqn:omega} simply
reads ${\ddot\eta}_{2}+\omega_{0}^{2}\eta_{2}=0$ where
$\eta_{2}=z_{4}\in\Real$. Hence the associated set of
characteristic frequencies is singleton
$\Omega_{2}=\{\omega_{0}\}$. The system associated to $\Gamma_{1}$
reads ${\ddot\eta}_{1}+(\omega_{0}^{2}I_{3}+R_{1})\eta_{1}=0$
where $\eta_{1}=[z_{1}\ z_{2}\ z_{3}]^{T}\in\Real^{3}$. The set of
eigenvalues of $R_{1}$ being $\{0,\,r,\,3r\}$, we have
$\Omega_{1}=\{\omega_{0},\,\sqrt{\omega_{0}^{2}+r},\,\sqrt{\omega_{0}^{2}+3r}\}$.
Now we can write $\Omega_{1}\cap\Omega_{2}=\{\omega_{0}\}$.
Therefore the second condition of Assumption~\ref{assume:one} is
also satisfied. Since the second and third conditions hold, the
first condition must not (because the oscillators do not
synchronize). The system
${\ddot\eta}_{2}+\omega_{0}^{2}\eta_{2}=0$ is clearly observable
from $z_{4}$. Therefore the other system
${\ddot\eta}_{1}+(\omega_{0}^{2}I_{3}+R_{1})\eta_{1}=0$ must be
unobservable from at least one of its states $z_{k}$,
$k\in\{1,\,2,\,3\}$. It can be shown that from $z_{2}$ the system
is indeed unobservable. This finding is not at all surprising when
we look at the solution depicted in Fig.~\ref{fig:unobs}.
\end{example}

\begin{remark}
The situation shown in Fig.~\ref{fig:unobs} not only renders the
first condition of Assumption~\ref{assume:one} more meaningful but
also suggests a refinement on it. If the system in
Fig.~\ref{fig:unobs} were slightly modified by relocating the
damper between the third and fourth pendulums (as opposed to the
original configuration where it connects the second and fourth
pendulums) the observability condition of
Assumption~\ref{assume:one} would still be violated yet the
pendulums would this time synchronize. The reason is that even the
component~\eqref{eqn:omega} described by the first three pendulums
is unobservable from the second pendulum, it nevertheless is
observable from the third. And the significance of the third
pendulum is that it is through it that the first component (in the
modified system) is connected via damper to the second component
(namely, to the fourth pendulum). One can carry this observation
further so as to suggest the following relaxation of the first
condition of Assumption~\ref{assume:one}: ``For all
$\ell=1,\,2,\,\ldots,\,c$ the system~\eqref{eqn:omega} is
observable from each $z_{k}$ for all
$k\in\{i:v_{i}\in\V_{\ell},\,d_{ij}\neq 0,\,j\notin\V_{\ell}\}$.''
\end{remark}

\begin{example}\label{ex:two}
Consider the following four coupled harmonic oscillators
\begin{eqnarray*}
&&{\ddot z}_{1}+\omega_{0}^{2}z_{1}+r(z_{1}-z_{2})=0\\
&&{\ddot z}_{2}+\omega_{0}^{2}z_{2}+d({\dot z}_{2}-{\dot z}_{3})+r(z_{2}-z_{1})=0\\
&&{\ddot z}_{3}+\omega_{0}^{2}z_{3}+d({\dot z}_{3}-{\dot z}_{2})+r(z_{3}-z_{4})=0\\
&&{\ddot z}_{4}+\omega_{0}^{2}z_{4}+r(z_{4}-z_{3})=0
\end{eqnarray*}
where $\omega_{0},\,d,\,r>0$. The associated $D$ and $R$ matrices
are provided below.
\begin{eqnarray*}
D=\left[\begin{array}{rrrr}0&0&0&0\\0&d&-d&0\\0&-d&d&0\\0&0&0&0\end{array}\right]
\,,\qquad
R=\left[\begin{array}{rrrr}r&-r&0&0\\-r&r&0&0\\0&0&r&-r\\0&0&-r&r\end{array}\right]\,.
\end{eqnarray*}
It turns out that for the pair $(D,\,R)$ condition~\eqref{eqn:PBH}
fails to hold and hence by Theorem~\ref{thm:one} the oscillators
do not synchronize. In particular,
\begin{eqnarray*}
{\rm null}\left[\begin{array}{c}R-2rI_{q}\\
D\end{array}\right] = {\rm range}
\left[\begin{array}{r}1\\-1\\-1\\1\end{array}\right]\not\subset
{\rm range} \left[\begin{array}{r}1\\1\\1\\1\end{array}\right]\,.
\end{eqnarray*}
The solution corresponding to the eigenvector $[1\ -1\ -1\ 1]^{T}$
is shown in Fig.~\ref{fig:extra} where the first and fourth
pendulums make a synchronized pair and the second and third
pendulums make another synchronized pair. These pairs oscillate
(with $\pi$ radians of phase difference between pairs) at
frequency $\omega=\sqrt{\omega_{0}^{2}+2r}$. (We note that the two
springs are identical.)
\begin{figure}[h]
\begin{center}
\includegraphics[scale=0.55]{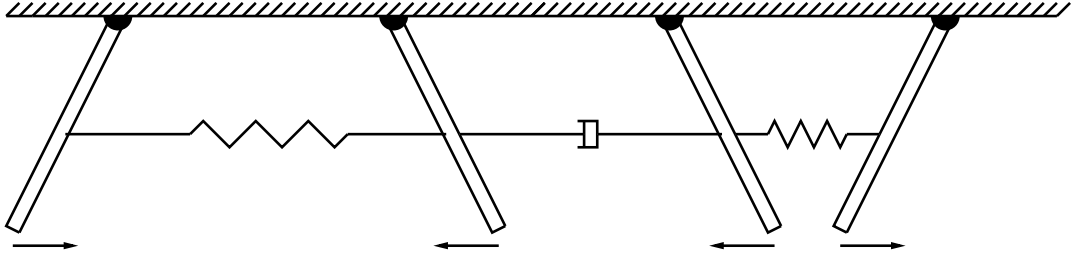}
\caption{Pendulums out of synchrony due to extra common
frequencies.}\label{fig:extra}
\end{center}
\end{figure}
Let us now figure out which condition(s) of
Assumption~\ref{assume:one} is violated here. The graphs
$\Gamma_{\Sigma}$ and $\Gamma_{\Delta}$ are given in
Fig.~\ref{fig:graphs2}. Since the graph $\Gamma_{\Sigma}$ is
connected, the third condition of Assumption~\ref{assume:one} is
satisfied.
\begin{figure}[h]
\begin{center}
\includegraphics[scale=0.55]{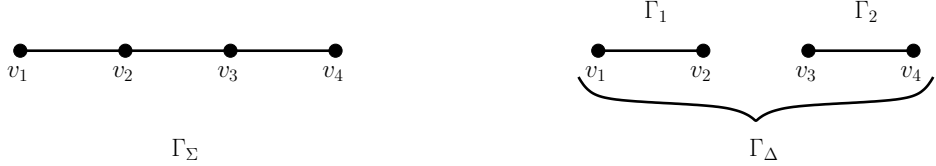}
\caption{The graphs associated to the coupled harmonic oscillators
in Example~\ref{ex:two}.}\label{fig:graphs2}
\end{center}
\end{figure}
Corresponding to the two components $\Gamma_{1}$ and $\Gamma_{2}$
are the matrices
\begin{eqnarray*}
R_{1}=R_{2}=\left[\begin{array}{rr}r&-r\\-r&r\end{array}\right]\,.
\end{eqnarray*}
For the graph $\Gamma_{1}$ the system~\eqref{eqn:omega} reads
${\ddot\eta}_{1}+(\omega_{0}^{2}I_{2}+R_{1})\eta_{1}=0$ where
$\eta_{1}=[z_{1}\ z_{2}]^{T}\in\Real^{2}$. It can be shown that
this system is observable from each $z_{k}$, $k=1,\,2$. Since
$R_{2}=R_{1}$, the same argument is valid also for the
system~\eqref{eqn:omega} associated to $\Gamma_{2}$. Therefore the
first condition of Assumption~\ref{assume:one} is also satisfied.
This implies that the second condition cannot hold. Let us verify
that the second condition does not hold. The set of eigenvalues of
$R_{1}$ being $\{0,\,2r\}$, we have
$\Omega_{1}=\{\omega_{0},\,\sqrt{\omega_{0}^{2}+2r}\}$. The
equality $R_{2}=R_{1}$ implies $\Omega_{2}=\Omega_{1}$. Hence
$\Omega_{1}\cap\Omega_{2}=\{\omega_{0},\,\sqrt{\omega_{0}^{2}+2r}\}\neq\{\omega_{0}\}$
as expected.
\end{example}

\section{Electrical networks}

\begin{figure}[h]
\begin{center}
\includegraphics[scale=0.55]{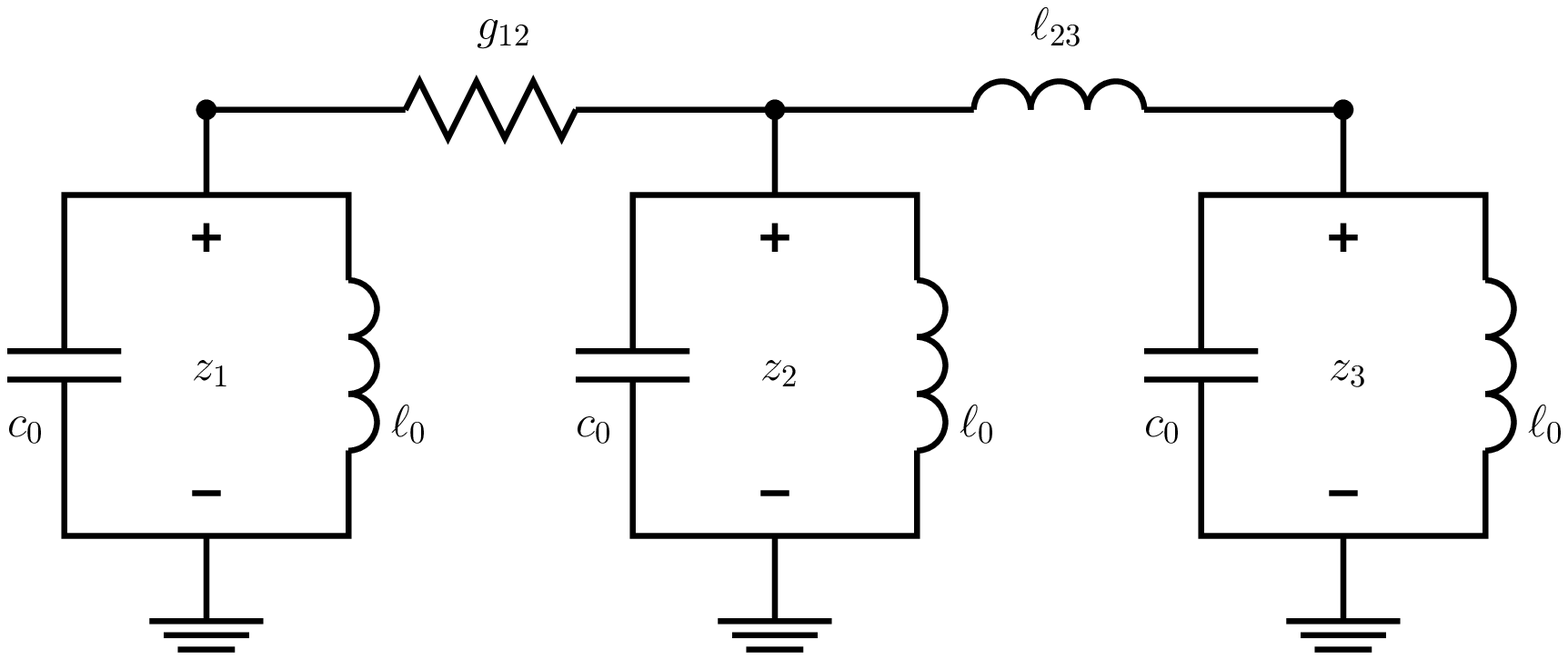}
\caption{Three coupled LC oscillators.}\label{fig:threeLC}
\end{center}
\end{figure}

Hitherto the coupled pendulums were the workhorse in our study of
the dynamics~\eqref{eqn:harmonic}. There are however other
important systems that share the same model; for instance, coupled
electrical oscillators. Consider an array of identical LC
oscillators where certain pairs $(i,\,j)$ are coupled via an LTI
resistor (with conductance $g_{ij}$) or an LTI inductor (with
inductance $\ell_{ij}$) or both; see Fig.~\ref{fig:threeLC}.
Writing Kirchhoff's Current Law (KCL) at each node then yields the
dynamics~\eqref{eqn:harmonic} in terms of {\em electrical}
parameters:
\begin{eqnarray}\label{eqn:harmonicLC}
c_{0}{\ddot z}_{i}+\ell_{0}^{-1}z_{i}+\sum_{j=1}^{q}g_{ij}({\dot
z}_{i}-{\dot z}_{j})+\sum_{j=1}^{q}h_{ij}(z_{i}-z_{j})=0\,,\qquad
i=1,\,2,\,\ldots,\,q
\end{eqnarray}
where $z_{i}$ denote the node voltages, $c_{0}>0$ and $\ell_{0}>0$
are respectively the capacitance and inductance of an individual
oscillator, and $h_{ij}=\ell_{ij}^{-1}$ when there is an inductor
(i.e., $\ell_{ij}\neq 0$) that connects the $i$th and $j$th nodes
and $h_{ij}=0$ otherwise. We work with passive components, i.e.,
$g_{ij}=g_{ji}\geq 0$ and $h_{ij}=h_{ji}\geq 0$. We take
$g_{ii}=0$ and $h_{ii}=0$. Note that the interconnection of such
an array can be represented by the admittance matrix
\begin{eqnarray}\label{eqn:Y}
Y(s) =
\left[\begin{array}{cccc}\sum_{j}y_{1j}(s)&-y_{12}(s)&\cdots&-y_{1q}(s)\\
-y_{21}(s)&\sum_{j}y_{2j}(s)&\cdots&-y_{2q}(s)\\
\vdots&\vdots&\ddots&\vdots\\
-y_{q1}(s)&-y_{q2}(s)&\cdots&\sum_{j}y_{qj}(s)
\end{array}\right]
\end{eqnarray}
where $y_{ij}(s)=g_{ij}+(\ell_{ij}s)^{-1}=g_{ij}+h_{ij}/s$ is the
admittance of the coupling between the $i$th and $j$th nodes.
Define the symmetric positive semidefinite matrices
\begin{eqnarray*}
G =
\left[\begin{array}{cccc}\sum_{j}g_{1j}&-g_{12}&\cdots&-g_{1q}\\
-g_{21}&\sum_{j}g_{2j}&\cdots&-g_{2q}\\
\vdots&\vdots&\ddots&\vdots\\
-g_{q1}&-g_{q2}&\cdots&\sum_{j}g_{qj}
\end{array}\right]\,,\qquad
H =
\left[\begin{array}{cccc}\sum_{j}h_{1j}&-h_{12}&\cdots&-h_{1q}\\
-h_{21}&\sum_{j}h_{2j}&\cdots&-h_{2q}\\
\vdots&\vdots&\ddots&\vdots\\
-h_{q1}&-h_{q2}&\cdots&\sum_{j}h_{qj}
\end{array}\right]\,.
\end{eqnarray*}
Note that $Y(s)=G+s^{-1}H$. Also note that $G/c_{0}$ and $H/c_{0}$
correspond to the matrices $D$ and $R$ of the
array~\eqref{eqn:harmonic}. In other words, $G$ represents the
dissipative coupling and $H$ the restorative coupling. Given a
matrix $A\in\Complex^{n\times n}$ let now $\lambda_{k}(A)$ denote
the $k$th smallest eigenvalue of $A$ with respect to the real
part. That is, ${\rm Re}\,\lambda_{1}(A)\leq{\rm
Re}\,\lambda_{2}(A)\leq\cdots\leq{\rm Re}\,\lambda_{n}(A)$. By
Corollary~\ref{cor:simple} we can then state that the LC
oscillators~\eqref{eqn:harmonicLC} synchronize if
$\lambda_{2}(G)>0$, i.e., if the dissipative coupling graph is
connected. Note that the condition $\lambda_{2}(G)>0$ is only
sufficient when $H\neq 0$. Now we point out an interesting
extension of this inequality, which turns out to manifest itself
in terms of the admittance matrix:

\begin{theorem}\label{thm:complexL}
The LC oscillators~\eqref{eqn:harmonicLC} synchronize if and only
if
\begin{eqnarray}\label{eqn:PBH3}
{\rm Re}\,\lambda_{2}(Y(j\omega))>0\ \mbox{for all}\ \omega>0\,.
\end{eqnarray}
\end{theorem}

\begin{proof}
Without loss of generality take $c_{0}=1$. Then $G$ and $H$
correspond to the matrices $D$ and $R$ of the
array~\eqref{eqn:harmonic}. Hence, by Theorem~\ref{thm:one}, the
LC oscillators~\eqref{eqn:harmonicLC} synchronize if and only if
\begin{eqnarray}\label{eqn:PBH4}
{\rm null}\left[\begin{array}{c}H-\lambda I_{q}\\
G\end{array}\right]\subset {\rm range}\,{\bf 1}_{q}\ \mbox{for
all}\ \lambda\in\Complex\,.
\end{eqnarray}
Our task therefore reduces to establishing the equivalence of
\eqref{eqn:PBH3} and \eqref{eqn:PBH4}, where
$Y(j\omega)=G+(j\omega)^{-1}H$. Let $\lambda\in\Complex$ be an
eigenvalue of $Y(j\omega)$ and $\xi\in\Complex^{q}$ be the
corresponding unit eigenvector, i.e., $Y(j\omega)\xi=\lambda\xi$
and $\|\xi\|^{2}=\xi^{*}\xi=1$, where $\xi^{*}$ is the conjugate
transpose of $\xi$. We can write
\begin{eqnarray*}
\lambda&=&\xi^{*}Y(j\omega)\xi\\
&=&\xi^{*}(G-j\omega^{-1}H)\xi\\
&=&\xi^{*}G\xi-j\xi^{*}H\xi/\omega\,.
\end{eqnarray*}
Since both $G$ and $H$ are symmetric positive semidefinite
matrices we have ${\rm Im}\,\lambda=-\xi^{*}H\xi/\omega\leq 0$ and
${\rm Re}\,\lambda=\xi^{*}G\xi\geq 0$. Therefore no eigenvalue of
$Y(j\omega)$ can be on the open left half-plane. Also note that by
construction $G{\bf 1}_{q}=0$ and $H{\bf 1}_{q}=0$. Therefore
$Y(j\omega){\bf 1}_{q}=0$ and we can let
$\lambda_{1}(Y(j\omega))=0$ for all $\omega>0$.

Suppose now \eqref{eqn:PBH3} fails. This means that ${\rm
Re}\,\lambda_{2}(Y(j\omega))=0$ for some $\omega>0$. There are two
possibilities, one of which is: (i) ${\rm
Im}\,\lambda_{2}(Y(j\omega))=0$. In this case the eigenvalue at
the origin is repeated and it must have at least two eigenvectors.
(Otherwise ${\bf 1}_{q}$ would be the only eigenvector for the
eigenvalue at the origin and there would exist a generalized
eigenvector $\xi\in\Complex^{q}$ satisfying $Y(j\omega)\xi={\bf
1}_{q}$. But such $\xi$ could not exist because it would lead to
the following contradiction: $q={\bf 1}_{q}^{T}{\bf 1}_{q}={\bf
1}_{q}^{T}Y(j\omega)\xi=(Y(j\omega){\bf 1}_{q})^{T}\xi=0$.)
Therefore we can find $\xi_{2}\notin{\rm range}\,{\bf 1}_{q}$
satisfying $Y(j\omega)\xi_{2}=0$. This implies
$0=\xi_{2}^{*}Y(j\omega)\xi_{2}=\xi_{2}^{*}G\xi_{2}-j\xi_{2}^{*}H\xi_{2}/\omega$.
Since both $G$ and $H$ are symmetric positive semidefinite
matrices we can deduce $G\xi_{2}=0$ and $H\xi_{2}=0$. That is,
${\rm null}\,G\cap{\rm null}\,H\supset{\rm span}\,\{\xi_{2}\}$.
Thus \eqref{eqn:PBH4} fails. Let us now consider the other
possibility: (ii) ${\rm Im}\,\lambda_{2}(Y(j\omega))<0$. Then we
can write $\lambda_{2}(Y(j\omega))=-j\beta_{2}$ for some
$\beta_{2}>0$. Let $\xi_{2}\in\Complex^{q}$ be the corresponding
unit eigenvector, i.e., $Y(j\omega)\xi_{2}=-j\beta_{2}\xi_{2}$ and
$\xi_{2}^{*}\xi_{2}=1$. Clearly, $\xi_{2}\notin{\rm range}\,{\bf
1}_{q}$. We can write $-j\beta_{2}=\xi_{2}^{*}Y(j\omega)\xi_{2}
=\xi_{2}^{*}G\xi_{2}-j\xi_{2}^{*}H\xi_{2}/\omega$. This yields
$\xi_{2}^{*}G\xi_{2}=0$. Consequently, $G\xi_{2}=0$ and
$-j\beta_{2}\xi_{2}=Y(j\omega)\xi_{2}=-jH\xi_{2}/\omega$.
Therefore $\xi_{2}$ has to be an eigenvector of $H$. In particular
we can write ${\rm null}\,G\cap{\rm null}\,(H-\beta_{2}\omega
I_{q})\supset{\rm span}\,\{\xi_{2}\}$ and \eqref{eqn:PBH4} once
again fails.

To show the other direction suppose this time that
\eqref{eqn:PBH4} fails. Then we can find an eigenvector
$\xi_{2}\notin{\rm range}\,{\bf 1}_{q}$ that satisfies
$G\xi_{2}=0$ and $H\xi_{2}=\beta_{2}\xi_{2}$ for some
$\beta_{2}\in\Real$. Note that $\beta_{2}$ has to be real because
it is an eigenvalue of $H$, a real symmetric matrix. Then we can
write
$Y(j\omega)\xi_{2}=(G-j\omega^{-1}H)\xi_{2}=-j\beta_{2}\omega^{-1}\xi_{2}$.
That is,  $\lambda=-j\beta_{2}\omega^{-1}$ is an eigenvalue of
$Y(j\omega)$. We also have $Y(j\omega){\bf 1}_{q}=0$. Therefore
${\rm Re}\,\lambda_{2}(Y(j\omega))=0$ and \eqref{eqn:PBH3} fails.
\end{proof}

\begin{remark}
Though a simple reexpression of Theorem~\ref{thm:one},
Theorem~\ref{thm:complexL} is nevertheless significant (from the
synchronization point of view) for it suggests a natural way of
combining the two different interconnection graphs: the
dissipative coupling graph and the restorative coupling graph. The
result is a single graph with complex-weighted edges whose
Laplacian is the admittance matrix $Y(j\omega)$.
\end{remark}

\begin{figure}[h]
\begin{center}
\includegraphics[scale=0.55]{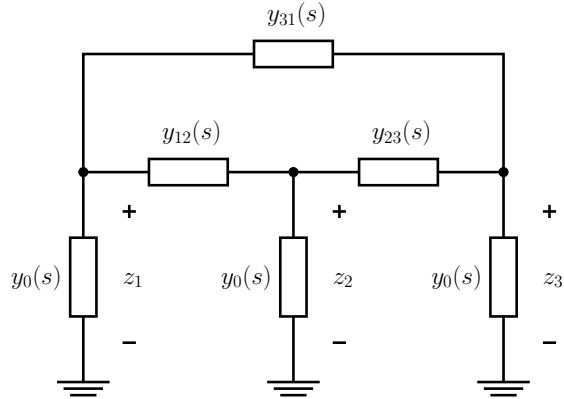}
\caption{Three coupled electrical
oscillators.}\label{fig:threeLC2n}
\end{center}
\end{figure}

We now briefly discuss how far the range of
condition~\eqref{eqn:PBH3} might extend. To this end we consider a
generalization of the dynamics~\eqref{eqn:harmonicLC} employing
the transfer function approach; see Fig.~\ref{fig:threeLC2n}. In
this general setting, each of $q$ identical
oscillators\footnote{Although we stick to the term {\em
oscillator}, the general setting we consider here allows also
systems that do not display oscillatory behavior, e.g., a single
capacitor.} is assumed to consist entirely of resistors,
capacitors, and inductors, all LTI and passive, i.e., with
positive resistance, capacitance, and inductance values. We denote
by $y_{0}(s)$ the admittance of an individual oscillator as seen
from a given pair of terminals. One of those terminals is
connected to the common ground and the other connects the
oscillator to the rest of the network. In short, we represent each
oscillator by an LTI passive one-port. As for coupling, the
connection between a pair $(i,\,j)$ of oscillators is through also
an LTI passive one-port (containing only resistors, capacitors,
and inductors) with admittance $y_{ij}(s)$. Note that
$y_{ij}(s)=y_{ji}(s)$. As before, we take $y_{ii}(s)=0$ and when
there is no direct connection between the pair $(i,\,j)$ we have
$y_{ij}(s)=0$. The overall interconnection gives us $Y(s)$, the
$q$-by-$q$ admittance matrix~\eqref{eqn:Y}. The array of coupled
oscillators, considered as a whole, we denote by
$\N(y_{0}(s),Y(s))$. The network $\N(y_{0}(s),Y(s))$ is said to
{\em synchronize} if the node voltages $z_{i}$ (see
Fig.~\ref{fig:threeLC2n}) synchronize, i.e.,
$|z_{i}(t)-z_{j}(t)|\to 0$ for all
$i,\,j\in\{1,\,2,\,\ldots,\,q\}$ and all initial conditions. (Note
that here the initial condition, which determines the future
evolution of the array, is the collection of all the initial
capacitor voltages and initial inductor currents throughout the
entire network.) In the sequel we will seek conditions
guaranteeing the synchronization of $\N(y_{0}(s),Y(s))$.

Thanks to passivity we will be able to proceed in our analysis
fairly rapidly. First, note that the eigenvalues that are
observable from the node voltages have to be the (finitely many)
roots of the {\em characteristic polynomial} $n(s)$ where
$n(s)/d(s)={\rm det}\,[y_{0}(s)I_{q}+Y(s)]$ and the polynomials
$n(s)$, $d(s)$ are coprime. Since the network is passive those
eigenvalues are confined to the closed left half-plane. Clearly,
the ones with strictly negative real parts do not play any role in
the steady state behavior of the network. This allows us to focus
on the eigenvalues on the imaginary axis. Suppose now
$\lambda=j\omega$ with $\omega\in\Real$ is such an eigenvalue.
Then (and only then) there exists a nonzero $\xi\in\Complex^{q}$
($\xi\in\Real^{q}$ if $\omega=0$) and $z(t)={\rm Re}(\xi
e^{j\omega t})$ is a possible trajectory that can be traced by the
node voltage vector $z=[z_{1}\ z_{2}\ \cdots\ z_{q}]^{T}$. KCL
imposes on this trajectory the constraint ${\rm
Re}\left([y_{0}(j\omega)I_{q}+Y(j\omega)]\xi e^{j\omega
t}\right)=0$. In other words,
\begin{eqnarray*}\label{eqn:purist}
\xi\in{\rm
null}\,[y_{0}(j\omega)I_{q}+Y(j\omega)]=:\setE(j\omega)\,.
\end{eqnarray*}
Therefore any steady state solution can be written as a sum of
finitely many terms $z(t)=\sum_{k}{\rm Re}(\xi_{k} e^{j\omega_{k}
t})$, where $\omega_{k}$ are distinct and
$\xi_{k}\in\setE(j\omega_{k})$ are nonzero. Evidently, this steady
state solution corresponds to a synchronized collection of node
voltages if and only if $\xi_{k}\in{\rm range}\,{\bf 1}_{q}$ for
all $k$. Note also that $\xi_{k}\in{\rm range}\,{\bf 1}_{q}$
implies $Y(j\omega_k)\xi_{k}=0$. Since
$\xi_{k}\in\setE(j\omega_{k})$, this means
$y_{0}(j\omega_{k})\xi_{k}=0$, i.e., $y_{0}(j\omega_{k})=0$. Hence
we obtained:

\begin{theorem}\label{thm:purist}
The network $\N(y_{0}(s),\,Y(s))$ synchronizes if and only if
\begin{eqnarray*}\label{eqn:purist}
{\rm null}\,[y_{0}(j\omega)I_{q}+Y(j\omega)]\subset {\rm
range}\,{\bf 1}_{q}\ \mbox{for all}\ \omega\in\Real\,.
\end{eqnarray*}
Also, for a synchronizing network, the steady state node voltages
have the form $z_{i}(t)=\sum_{k}{\rm
Re}(\alpha_{k}e^{j\omega_{k}t})$ with $\alpha_{k}\in\Complex$ and
$\omega_{k}\in\Real$ satisfying $y_{0}(j\omega_{k})=0$.
\end{theorem}

Purists may rightfully contend that the subspace $\setE(j\omega)$
is not always well-defined because for certain frequencies
$\omega$ either $y_{0}(j\omega)$ or some entries $y_{ij}(j\omega)$
of $Y(j\omega)$ may attain infinite magnitude. This however is
only a minor mathematical obstacle, easy to circumvent by thinking
in terms of the physical system that the model stands for.
Consider the case $|y_{0}(j\omega)|=\infty$. This means that the
impedance $y_{0}^{-1}(j\omega)$ is zero, i.e., the oscillators
behave as short circuit at that particular frequency $\omega$. As
a result, all nodes are grounded, i.e., all the node voltages
$z_{i}$ have to be zero, meaning $\setE(j\omega)=\{0\}$. Consider
now the other potentially ambiguous case,
$|y_{ij}(j\omega)|=\infty$ for certain pairs $(i,\,j)$ while
$|y_{0}(j\omega)|<\infty$. Note that $|y_{ij}(j\omega)|=\infty$
implies that the nodes $i$ and $j$ are short-circuited, i.e.,
$z_{i}=z_{j}$. To get rid of the infinite terms
$y_{ij}(j\omega)=y_{ji}(j\omega)$ in the expression ${\rm
null}\,[y_{0}(j\omega)I_{q}+Y(j\omega)]$ we can remove $i$th and
$j$th rows from the matrix $[y_{0}(j\omega)I_{q}+Y(j\omega)]$ and
inject the following two new rows: (i) the sum of the removed pair
of rows and (ii) a row that imposes the equality $z_{i}=z_{j}$.
The cure can be repeated until all the infinite terms are gone.
Let us demonstrate the procedure on an example network with $q=4$
nodes whose admittance matrix reads
\begin{eqnarray*}
Y(s)=\left[\begin{array}{cccc}
y_{12}(s)+y_{14}(s)&-y_{12}(s)&0&-y_{14}(s)\\
-y_{12}(s)&y_{12}(s)+y_{23}(s)&-y_{23}(s)&0\\
0&-y_{23}(s)&y_{23}(s)+y_{34}(s)&-y_{34}(s)\\
-y_{14}(s)&0&-y_{34}(s)&y_{14}(s)+y_{34}(s)
\end{array}\right]
\end{eqnarray*}
Suppose that at some frequency $\omega$ the admittances
$y_{12}(j\omega)$ and $y_{23}(j\omega)$ are infinite. This gives
us the equalities $z_{1}=z_{2}$ and $z_{2}=z_{3}$. Summing up the
first three rows of $[y_{0}(j\omega)I_{q}+Y(j\omega)]$ lets us get
rid of the terms $y_{12}(j\omega)$ and $y_{23}(j\omega)$. Then we
inject the extra rows $[1\ -1\ 0\ 0]$ and $[0\ 1\ -1\ 0]$ to
represent the relation $z_{1}=z_{2}=z_{3}$. Hence we can express
$\setE(j\omega)$ as
\begin{eqnarray*}
\setE(j\omega)={\rm null}\left[\begin{array}{cccc}
y_{0}(j\omega)+y_{14}(j\omega)&y_{0}(j\omega)&y_{0}(j\omega)+y_{34}(j\omega)&-y_{14}(j\omega)-y_{34}(j\omega)\\
-y_{14}(j\omega)&0&-y_{34}(j\omega)&y_{0}(j\omega)+y_{14}(j\omega)+y_{34}(j\omega)\\
1&-1&0&0\\
0&1&-1&0
\end{array}\right]
\end{eqnarray*}
with righthand side cleansed of the infinite terms. Incidentally,
another relevant point we want to make has to do with the
eigenvalues of $Y(j\omega)$. Since the roots of the polynomial
$p(\lambda)={\rm det}\,[-\lambda I_{q}+Y(j\omega)]$ are the
eigenvalues of $Y(j\omega)$, we can use the procedure described
above to {\em define} the finite eigenvalues of $Y(j\omega)$ when
some of its entries are infinite. For instance, for the previous
example, the finite eigenvalues of $Y(j\omega)$, when
$y_{12}(j\omega)$ and $y_{23}(j\omega)$ are infinite, are defined
as the roots of the polynomial
\begin{eqnarray*}
p(\lambda)={\rm det}\left[\begin{array}{cccc}
-\lambda+y_{14}(j\omega)&-\lambda&-\lambda+y_{34}(j\omega)&-y_{14}(j\omega)-y_{34}(j\omega)\\
-y_{14}(j\omega)&0&-y_{34}(j\omega)&-\lambda+y_{14}(j\omega)+y_{34}(j\omega)\\
1&-1&0&0\\
0&1&-1&0
\end{array}\right]\,.
\end{eqnarray*}
Being thus able to single out the finite eigenvalues allows us to
continue to use the notation $\lambda_{k}(Y(j\omega))$, which will
henceforth stand for the $k$th smallest finite eigenvalue of
$Y(j\omega)$ with respect to the real part.

Consider now a network $\N(y_{0}(s),\,Y(s))$ that does not
synchronize. For this network Theorem~\ref{thm:purist} assures us
that there exist a vector $\xi\notin{\rm range}\,{\bf 1}_{q}$ and
a frequency $\omega\in\Real$ satisfying
$Y(j\omega)\xi=-y_{0}(j\omega)\xi$ and $|y_{0}(j\omega)|<\infty$.
Therefore $\lambda=-y_{0}(j\omega)$ is a (finite) eigenvalue of
$Y(j\omega)$. Since the network is passive, all the eigenvalues of
$Y(j\omega)$ belong to the closed right half-plane, i.e., ${\rm
Re}\,\lambda_{k}(Y(j\omega))\geq 0$ for all $k$, whence ${\rm
Re}\,\lambda\geq 0$. Also, again due to passivity, ${\rm
Re}\,y_{0}(j\omega)\geq 0$, whence ${\rm Re}\,\lambda\leq 0$.
Consequently, ${\rm Re}\,\lambda = 0$. This implies, since
$Y(j\omega){\bf 1}_{q}=0$, the matrix $Y(j\omega)$ has at least
two eigenvalues on the imaginary axis. This allows us to assert
${\rm Re}\,\lambda_{2}(Y(j\omega))=0$. To summarize:

\begin{corollary}
The network $\N(y_{0}(s),\,Y(s))$ synchronizes if ${\rm
Re}\,\lambda_{2}(Y(j\omega))>0$ for all $\omega\in\Real$.
\end{corollary}

\section{Conclusion}

In this paper we studied the synchronization of identical (linear)
pendulums coupled via dampers and springs. We first presented a
necessary and sufficient condition for synchronization and then
pointed out a sufficient set of conditions that may occasionally
turn out to be easier to check than the former. Toward the end of
the paper we applied the results obtained for pendulums to
understanding better the collective behavior of coupled
oscillators in LTI passive electrical networks. In particular, we
established a relation between the second smallest eigenvalue of
the node admittance matrix and the tendency of the individual
systems to oscillate in unison.

\bibliographystyle{plain}
\bibliography{references}

\begin{thebibliography}{10}

\bibitem{adato13}
R.~Adato, A.~Artar, S.~Erramilli, and H.~Altug.
\newblock Engineered absorption enhancement and induced transparency in coupled
  molecular and plasmonic resonator systems.
\newblock {\em Nano Letters}, 13:2584--2591, 2013.

\bibitem{arnold89}
V.I. Arnold.
\newblock {\em Mathematical Methods of Classical Mechanics (Second Edition)}.
\newblock Springer, 1989.

\bibitem{cai10}
C.~Cai and S.E. Tuna.
\newblock Synchronization of nonlinearly coupled harmonic oscillators.
\newblock In {\em Proc. of the American Control Conference}, pages 1767--1771,
  2010.

\bibitem{fermi55}
E.~Fermi, J.~Pasta, and S.~Ulam.
\newblock Studies on non linear problems.
\newblock {\em Los Alamos Document LA-1940}, 1955.

\bibitem{kapitaniak14special}
T.~Kapitaniak and J.~Kurths (Eds).
\newblock Synchronized pendula: {F}rom {H}uygens' clocks to chimera states.
\newblock {\em European Physical Journal Special Topics}, 223:609--612, 2014.

\bibitem{kapitaniak14}
T.~Kapitaniak, P.~Kuzma, J.~Wojewoda, K.~Czolczynski, and Y.~Maistrenko.
\newblock Imperfect chimera states for coupled pendula.
\newblock {\em Scientific Reports}, 4:6379, 2014.

\bibitem{marcheggiani14}
L.~Marcheggiani, R.~Chacon, and S.~Lenci.
\newblock On the synchronization of chains of nonlinear pendula connected by
  linear springs.
\newblock {\em European Physical Journal Special Topics}, 223:729--756, 2014.

\bibitem{ren08}
W.~Ren.
\newblock Synchronization of coupled harmonic oscillators with local
  interaction.
\newblock {\em Automatica}, 44:3195--3200, 2008.

\bibitem{su13}
H.~Su, M.Z.Q. Chen, X.~Wang, H.~Wang, and N.V. Valeyev.
\newblock Adaptive cluster synchronisation of coupled harmonic oscillators with
  multiple leaders.
\newblock {\em IET Control Theory and Applications}, 7:765--772, 2013.

\bibitem{su09}
H.~Su, X.~Wang, and Z.~Lin.
\newblock Synchronization of coupled harmonic oscillators in a dynamic
  proximity network.
\newblock {\em Automatica}, 45:2286--2291, 2009.

\bibitem{sun14}
W.~Sun, J.~Lu, S.~Chen, and X.Yu.
\newblock Synchronisation of directed coupled harmonic oscillators with
  sampled-data.
\newblock {\em IET Control Theory and Applications}, 8:937--947, 2014.

\bibitem{sun15}
W.~Sun, X.~Yu, J.~Lu, and S.~Chen.
\newblock Synchronization of coupled harmonic oscillators with random noises.
\newblock {\em Nonlinear Dynamics}, 79:473--484, 2015.

\bibitem{zhang12}
H.~Zhang and J.~Zhou.
\newblock Synchronization of sampled-data coupled harmonic oscillators with
  control inputs missing.
\newblock {\em Systems \& Control Letters}, 61:1277--1285, 2012.

\bibitem{zhou12}
J.~Zhou, H.~Zhang, L.~Xiang, and Q.~Wu.
\newblock Synchronization of coupled harmonic oscillators with local
  instantaneous interaction.
\newblock {\em Automatica}, 48:1715--1721, 2012.

\end{thebibliography}
\end{document}